%% file: inv-hoch.tex
\documentclass[12pt, a4paper,pdftex]{amsart} 
\input{preamble.tex}

\title{On the Hochschild homology of involutive algebras}

\author{Rams\`es Fern\`andez-Val\`encia} 
\email{ramses.fernandez.valencia@gmail.com}

\author{Jeffrey Giansiracusa} 
\email{j.h.giansiracusa@swansea.ac.uk} 
\address{Department of Mathematics, Swansea University \\ Singleton Park \\ Swansea SA2 8PP, UK}

\newcommand{\ihomA}{\Hom_{A\mhyphen i\Bimod}} 
\newcommand{\homA}{\Hom_{A\mhyphen\Bimod}} 
\newcommand{\Aie}{A^{ie}}

\date{\today}

\begin{document}
\begin{abstract}
  We study the homological algebra of bimodules over involutive associative algebras.  We show that
  Braun's definition of involutive Hochschild cohomology in terms of the complex of
  involution-preserving derivations is indeed computing a derived functor: the $\Z/2$-invariants
  intersected with the center.  We then introduce the corresponding involutive Hochschild homology
  theory and describe it as the derived functor of the pushout of $\Z/2$-coinvariants and abelianization.
\end{abstract}
\maketitle

\section{Introduction}

Hochschild cohomology is a cohomology theory for associative algebras that describes their
deformation theory.  Under mild hypotheses, the groups $HH^*(A,A)$ can be defined in any of the
following equivalent ways:
\begin{enumerate}
\item the homology of the usual Hochschild cochain complex of $A$,
\item the homology of the complex of coderivations on the tensor coalgebra of $\Sigma A$ (or
  equivalently, the complex of continuous derivations on the completed tensor algebra of $\Sigma^{-1}A^\vee$).
\item the derived center of $A$,
\end{enumerate}
There is a corresponding Hochschild \emph{homology} theory that can be defined as the derived
abelianization of $A$ or by writing down the usual Hochschild chain complex.  The derived functor
description is perhaps most fundamental and it is based on the fact that $A$-bimodules form an
abelian category that can equivalently be described as the category of right modules over the
enveloping algebra $A^e = A\otimes A^{op}$. 

In \cite{Braun}, Braun studied the Hochschild theory of involutive algebras (and
$A_\infty$-algebras), meaning algebras equipped with a map $a\mapsto a^*$ such that $a^{**}=a$ and
$(ab)^*=b^*a^*$.  He introduced an involutive variant of Hochschild cohomology by restricting to the
subcomplex of the derivation complex consisting of derivations that commute with the involution.
The ordinary Hochschild cohomology of an involutive algebra splits as a sum of this involutive
Hochschild cohomology and a skew factor (assuming the characteristic of the ground field is not 2). 

The purpose of this short note is to develop enough homological algebra for bimodules over
involutive algebras to give a derived functor description of Braun's involutive Hochschild
cohomology.  From this perspective we are also able to define the corresponding involutive
Hochschild \emph{homology} theory.  One key novel feature of the involutive theory is that it is
based on the abelian category of \emph{involutive bimodules}.  In contrast with the non-involutive case,
involutive bimodules are actually equivalent to modules over a certain semidirect product of the
enveloping algebra with the group ring $k[\Z/2]$.

Our motivation for studying involutive Hochschild theory comes from the first author's work on
unoriented topological conformal field theories.  Costello \cite{Costello} showed that an open 2d
oriented TCFT is essentially a Calabi-Yau $A_\infty$-algebra, and such a theory admits a universal
extension to an open-closed theory with closed state space (the value of the functor on a circle)
given by the Hochschild chain complex of the algebra of the open theory.  In \cite{Ramses}, this
picture is extended to Klein (i.e., unoriented) 2d TCFTs: open theories now correspond to involutive
Calabi-Yau $A_\infty$-algebras, and the closed state space of the universal open-closed extension
turns out to be the involutive Hochschild chain complex of the open state algebra.

\subsection*{Acknowledgements}
Both authors were supported by EPSRC grant EP/I003908/2.

\section{Involutive algebras and their bimodules}

\subsection{Involutive algebras}

Let $k$ be a field. An \emph{involutive vector space} is a vector space $V$ (assumed to be over $k$)
equipped with an automorphism of order 2, which we will usually write as $v\mapsto v^*$.  I.e., it
is a representation of the cyclic group $\Z/2$.  We let $i\Vect_k$ denote the category of involutive
$k$-vector spaces and linear maps that commute with the involutions.

An \emph{involutive $k$-algebra} is an involutive vector space $A$ equipped with an associative and
unital multiplication map $A\otimes_k A \to A$ such that
\[
(ab)^* = b^*a^*
\] 
for any $a,b\in A$.  Note that it follows automatically that $1^*= 1$ and $0^*=0$.

\begin{remark}
  An associative $k$-algebra is the same as a monoid in the monoidal category of vector spaces with
  tensor product.  Although involutive vector spaces are the same as $\Z/2$ representations, and the
  tensor product $\otimes_k$ gives this category a monoidal structure, involutive algebras are
  \emph{not} the same as monoids in the monoidal category of $\Z/2$ representations; they are a
  restricted class of such monoids.  The tensor product $\otimes_{\Z/2}$ also provides a monoidal
  product on the category, but involutive algebras are not monoids for this structure.
\end{remark}

\begin{example}\label{ex:inv-algebras}
\begin{enumerate}
\item Any commutative algebra $A$ becomes an involutive algebra when equipped with the identity as
  involution.  More generally, any $k$-algebra map of order 2 fixing 1 makes $A$ an involutive algebra.

\item Let $V$ be an involutive vector space and let $TV = \bigoplus_n V^{\otimes n}$ be the tensor
  algebra on $V$.  The tensor algebra becomes an involutive algebra with involution given by \[(v_1
  \otimes \cdots \otimes v_n)^* = v_n^* \otimes \cdots \otimes v_1^*.\]

\item  Let $G$ be a discrete group.  The group ring $k[G]$ is an involutive $k$-algebra with involution
  given by $g\mapsto g^{-1}$.  
\end{enumerate}
\end{example}

\subsection{Involutive bimodules}

First suppose that $A$ is an associative $k$-algebra. An $A$-bimodule $M$ is a $k$-vector space with
left and right multiplication maps  $A\otimes_k M \to M$ and $M\otimes_k A \to M$ that commute:
$(a\cdot m) \cdot b = a\cdot (m\cdot b)$ for all $a,b \in A$ and $m\in M$.  In category-theoretic
terms, $M$ is a bimodule for the monoid $A$ in the monoidal category $(\Vect_k, \otimes_k)$.
Equivalently, an $A$-bimodule is the same as a left module over the enveloping algebra $A^e =
A\otimes_k A^{op}$.

Now let $A$ be an involutive $k$-algebra.  An \emph{involutive $A$-bimodule} is a bimodule equipped
with an involution satisfying the compatibility condition between the left and right actions and the involution.
\[
(a\cdot m)^* = m^* \cdot a^*.
\] 
Note that, unlike the non-involutive case, here the left and right $A$-module structures determine
each other, so an involutive bimodule is determined by a vector space equipped with both a left
$A$-module structure and an involution, but there is a compatibility condition that these two
structures must satisfy coming from the fact that the left and right $A$-module structures on a
bimodule commute.  This condition is:
\begin{equation}\label{eq:LR-compat}
b\cdot (a \cdot m^*)^*  = b\cdot (m \cdot a^*) = (b\cdot m)\cdot a^* = (a \cdot (b \cdot m)^*)^*
\end{equation}
for $a,b\in A$ and $m$ in an involutive bimodule $M$.

One can describe the category of involutive bimodules as a category of left modules as
  follows.  Consider the algebra $\Aie:= A^e \otimes k[\Z/2]$ with product defined by
\[
(x\otimes\tau^i) \cdot (y\otimes \tau^j) = (x\cdot\tau^i(y))\otimes \tau^{i+j},
\]
where $\tau$ is the generator of $\Z/2$ and it acts on $A^e = A\otimes A^{op}$ by
$\tau(a\otimes b) = b^* \otimes a^*$.  We call $\Aie$ the \emph{involutive enveloping algebra of
  $A$}.

\begin{proposition}
There is an equivalence of categories
\[
A\mhyphen i\Bimod \cong \Aie\mhyphen \Mod.
\]
\end{proposition}
\begin{proof}
  The subalgebra of $\Aie$ consisting of elements of the form $x\otimes 1$ is isomorphic to
  $A^e$.  Given an $\Aie$-module $M$, the action of $A^e \subset \Aie$ defines an
  $A$-bimodule struture on $M$ as usual.  The action of the subalgebra $k[\Z/2]$ defines an
  involution on $M$, and this in fact yields an involutive bimodule since, by the associativity
  of the $\Aie$-action, multiplying by
\[
(1\otimes 1)\otimes \tau \cdot (a \otimes b) \otimes \tau = (b^* \otimes a^*) \otimes 1
\]
is equal to multiplying first by $(a \otimes b) \otimes \tau$ and then by
$(1\otimes 1)\otimes \tau$.  In terms of the induced bimodule structure and involution on $M$,
this says that $(b^* m a^*)$ is equal to $(am^*b)^*$, and hence $M$ becomes an involutive bimodule.

Conversely, if $M$ is an involutive bimodule then it becomes an $A^e$-module, its involution makes
it a module over $k[\Z/2]$, and the compatibility relation $(a m b)^* = b^* m a^*$ says that the
involution and bimodule structure combine to define an associative action of $\Aie$.
\end{proof}

One sees that the forgetful functor
\[
A\mhyphen i\Bimod \to A\mhyphen\Bimod.
\]
is faithful; however, it fails to be conservative (meaning that there are
involutive bimodules $M$ and $N$ that are not isomorphic in $A\mhyphen i\Bimod$, but they become
isomorphic in $A\mhyphen \Bimod$), and hence it is not full.  The following simple example illustrates this. 

\begin{example}
Let $A=k$ with the trivial involution, so $A$-bimodules are just $k$-vector spaces, and involutive
$A$-bimodules are just involutive vector spaces.  Let $V=k^2$ with the trivial involution, and let
$W=k^2$ with involution $(x,y)^* = (y,x)$.   As bimodules (i.e., vector spaces), $V$ and $W$ are clearly isomorphic,
but as involutive bimodules (i.e., involutive vector spaces) they are not.
\end{example}

If $M$ and $N$ are involutive $A$-bimodules, then we write $\ihomA(M,N)$ for the set of involutive
$A$-bimodule homomorphisms from $M$ to $N$, which is to say the set of bimodule homomorphisms that
commute with the involutions.  Both this and the set of bimodule homomophisms are $k$-vector spaces
and there is a natural linear inclusion map
\[
\ihomA(M,N) \hookrightarrow \homA(M,N).
\] 
However, the vector space $\ihomA(M,N)$ also carries an involution $f \mapsto f^*$ defined by
\[
f^*(m) = f(m)^*, \text{ or equivalently, } f(m^*).
\]

\subsection{Some functors and adjunctions}

Let $A$ be a $k$-algebra.  We first recall the adjunction between $A$-bimodules and vector
spaces. If $M$ is an $A$-bimodule then we may consider the functor
\[
\homA(M,-): A\mhyphen \Bimod \to \Vect_k.
\]
Note that in the special case of $M=A\otimes_k A$ with the bimodule structure given by 
\[a_1 \cdot (a_2 \otimes a_3) \cdot a_4 = a_1a_2 \otimes a_3 a_4,\]
the functor $\homA(A\otimes_k A, -)$ coincides with the forgetful functor sending a bimodule to its
underlying vector space.

If $V$ is a vector space and $M$ is an $A$-bimodule then the vector spaces $M\otimes_k V$ and
$\Hom_k(M,V)$ have canonical $A$-bimodule structures induced from the bimodule structure on $M$.
The functor \[M \otimes_k (-): \Vect_k \to A\mhyphen\Bimod\]
is left adjoint to $\homA(M,-)$.  A \emph{free} bimodule is a bimodule in the essential image of
$ (A\otimes_k A)\otimes_k (-)$.  When viewed as $A^e$-modules, they are free modules.  In section
\ref{sec:free-and-projective} below we will discuss an analogous notion of free involutive bimodules.

We now turn to the involutive variant of the above.  First suppose $V$ and $W$ are involutive vector
spaces.  While $V\otimes W$ has three involutions to choose from (from the involution on $V$, the
involution on $W$, or both at the same time), the quotient $V\otimes_{\Z/2} W$ inherits a canonical
involution 
\[
v\otimes w \mapsto v^* \otimes w = v\otimes w^*
\]
(the involution on $V$ is identified with the involution on $W$, and doing both involutions
simultaneously becomes the identity). This is a special case of the fact that the tensor product of
$R$-modules is again $R$-module when $R$ is a commutative ring; here $R$ is the group ring
$k[\Z/2]$.

Now let $A$ be an involutive algebra and $M$ an
involutive $A$-bimodule.  We can regard $\ihomA(M,-)$ as a functor
$A\mhyphen i\Bimod \to i\Vect_k$.  Given an involutive vector space $V$, the involutive vector space 
\[
M\otimes_{\Z/2} V
\]
becomes an involutive $A$-bimodule since
\begin{align*}
(m\otimes v)^* \cdot a^* & = (m^* \otimes v) \cdot a^* \\
                                         & = (m^* \cdot a^*)\otimes v \\
                                         & = (a\cdot m)^*\otimes v \\
                                         & = ((a\cdot m)\otimes v)^* = (a\cdot (m\otimes v))^*.
\end{align*}

\begin{proposition}
Let $A$ be an involutive algebra and $M$ and involutive $A$-bimodule.
There is an adjuction of functors 
\[
M \otimes_{\Z/2} (-): i\Vect \leftrightarrows A\mhyphen i\Bimod : \ihomA(M,-).
\]
\end{proposition}
\begin{proof}
  Let $L$ be an involutive $A$-bimodule and $V$ an involutive vector space.  A morphism of
  $A$-bimodules $f: M\otimes_k V \to L$ is adjoint to a morphism of vector spaces
  $g: V \to \homA(M,L)$.  Now we claim that $f$ descends to a morphism of involutive bimodules
  $M\otimes_{\Z/2} V \to L$ if and only if $g$ factors through a morphism of involutive vector spaces
\[
\widetilde{g}: V \to \ihomA(M,L).
\]
I.e., we claim that 
\[
f(m\otimes v^*) = f(m^* \otimes v) = f(m \otimes v)^*
\]
 if and only if
\[
g(v)(m)^* = g(v)(m^*) = g(v^*)(m).
\]
 To see this, first suppose that $f$ descends to an involutive morphism.  Then we have
\begin{align*}
g(v)(m)^* & = f(m\otimes v)^* \\
 & = f(m\otimes v^*) = f(m^* \otimes v)\\
 & = g(v^*)(m) = g(v)(m^*),
\end{align*}
where the equalities on the first and third lines are due to $f$ and $g$ being adjoint, and the
equalities on the second line come from the hypotheses on $f$.  The verification of the other
direction is just a permutation of the above sequence of steps.
\end{proof}

We now turn our attention to the functor $\otimes_{\Aie}$.  Given involutive bimodules $M$
and $N$, $M\otimes_{\Aie} N$ is \textit{a priori} a vector space. It can be described as
the quotient of $M\otimes_{A^e} N$ by the vector subspace spanned by the elements
\[m^* \otimes n - m \otimes n^*\]
for $m\in M$ and $n\in N$, and so as with $M \otimes_{\Z/2} N$, it carries an involution:
$(m\otimes n)^* = m^*\otimes n = m \otimes n^*$.

\begin{proposition}\label{prop:Ae-to-Aie-quotient}
  Given involutive bimodules $M$ and $N$, let $\Z/2$ act on $M\otimes_A N$ by
\[
m\otimes n \mapsto m^* \otimes n^*.
\]
Then there is an isomorphism of vector spaces
\[
(M\otimes_A N)_{\Z/2} \cong M\otimes_{\Aie} N.
\]
\end{proposition}
\begin{proof}
In $(M\otimes_A N)_{\Z/2}$ 
we have
\[
[am \otimes n] = [(am)^* \otimes n^*] = [m^* a^* \otimes n^*] = [m^* \otimes a^* n^*] = [m \otimes na],
\]
so $(M\otimes_A N)_{\Z/2}$ is a quotient of $M\otimes_{\Aie} N$.  On the other hand,
$M\otimes_{\Aie} N$ is clearly a quotient of $(M\otimes_A N)_{\Z/2}$, and so the two are isomorphic.
\end{proof}

Summing up, we have functors 
\begin{align*}
 (-) \otimes_{\Z/2} (-) : \:& i\Vect_k \times i\Vect_k \to i\Vect_k \\
\text{or}\quad &  A\mhyphen i\Bimod \times  i\Vect_k \to  A\mhyphen i\Bimod,\\
\end{align*}
and
\[
 (-) \otimes_{\Aie} (-) :  A\mhyphen i\Bimod \times  A\mhyphen i\Bimod \to  i\Vect_k.
\]

\subsection{Center and Abelianization}

We first recall the non-involutive setup.  The \emph{center} of an $A$-bimodule $M$ is the vector
subspace
\[
Z(M) := \{m\in M \:\: | \:\: am = ma \text{ for all } a\in A\} \subset M;
\]
it is a bimodule over the center of $A$, and is naturally isomorphic to $\homA (A,M)$.  The
\emph{abelianization} of $M$ is the quotient vector space 
\[
Ab(M) := M/(am \sim ma \:\: | \:\: m\in M, a\in A),
\]
which canonically has the structure of a $A$-bimodule and is naturally isomorphic to $A\otimes_{A^e} M$.

We now turn to the case of involutive bimodules.  Let $A$ be an involutive algebra and $M$ an
involutive $A$-bimodule.  We define the \emph{involutive center} of $M$ to be the involutive vector space
\[
iZ(M) := \ihomA(A,M),
\]
and we define the \emph{involutive abelianization} of $M$ to be the involutive vector space
\[
iAb(M) := A\otimes_{\Aie} M.
\]

\begin{proposition}
The involutive center of $M$ is naturally isomorphic to the pullback (i.e., intersection) of the involutive
vector spaces
\[
Z(M) \hookrightarrow M \hookleftarrow M^{\Z/2}.
\]
The involutive abelianization of $M$ is naturally isomorphic to the pushout of the involutive vector
spaces
\[
Ab(M) \leftarrow M \to M_{\Z/2}.
\]
\end{proposition}
\begin{proof}
A morphism of involutive $A$-bimodules $f:A \to M$ is entirely determined by $f(1)$, which must be
an element in $Z(M)$ since $a \cdot f(1) = f(a) = f(1) \cdot a$ for any $a\in A$, and must be fixed
by the involution since $1\in A$ is fixed.  This shows that $iZ(M)$ is contained in $Z(M)\cap
M^{\Z/2}$.  Conversely, sending $1\in A$ to any element in this intersection uniquely extends to a
well-defined bimodule morphism that clearly commutes with the involutions.

Now consider the pushout $P$ of $Ab(M) \leftarrow M \to M_{\Z/2}$.  First observe that, by the
universal property of the pushout, there is a natural map $P \to A \otimes_{\Aie} M$ sending
$[m]$ to $[1\otimes m]$.  An inverse to this should send $[a\otimes m]$ to $[am]$, and it remains to
check that this is well defined.  This formula gives a map $f: A\otimes_{A^e} M \to P$, and it
satisfies
\[
f(a^*\otimes m) = [a^*m] = [ma^*] = [am^*] = f(a \otimes m^*),
\]
so it descends to $A\otimes_{\Aie} M$, giving the desired inverse.
\end{proof}

\section{Homological algebra}

Let $A$ be an involutive $k$-algebra.  The categories of involutive vector spaces and involutive
$A$-bimodules are abelian categories; this follows immediately from the identifications as module categories,
\[
i\Vect_k \cong k[\Z/2]\mhyphen \Mod \quad \text{and} \quad A\mhyphen i\Bimod \cong \Aie\mhyphen \Mod.
\]
Hence we may talk about projective objects, chain complexes, and quasi-isomorphisms in each of these
categories.  In this section we will show that if $A$ is projective as an involutive vector space
then the usual construction of the bar resolution in fact provides a projective, and hence flat,
resolution of $A$ in the category of involutive bimodules.

\subsection{Flat and projective involutive bimodules}\label{sec:free-and-projective}

Projective objects in $i\Vect_k$ and $A\mhyphen i\Bimod$ are defined by the usual lifting property.
As these are module categories, the usual characterization holds: an involutive vector space is
projective if, viewed as a $k[\Z/2]$-module, it is a direct sumand of a free module, and an
involutive bimodule is projective if and only if, when viewed as a $\Aie$-module, it is a
direct summand of a free module. For the purposes of this paper we will not need to give a more
concrete characterization of projective involutive bimodules.

\begin{remark}
  If the characteristic of $k$ is different from 2 then every finite dimensional involutive vector
  space (i.e., $\Z/2$-representation) is projective by Maschke's Theorem.  In characteristic 2 a
  finite dimensional involutive vector space is projective if and only if it does not contain
  the trivial representation of $\Z/2$ as a direct summand.  This is because every finite dimensional
  $\Z/2$-representation splits as a sum of copies of the trivial representation and the regular
  representation, which is indecomposable.  While the trivial representation is a subrepresentation
  of the regular representation, it is not a direct summand, nor is it a summand of any number of
  copies of the regular representation, and hence it is not projective, nor is anything that contains it
  as a summand.
\end{remark}

\begin{proposition}\label{prop:free-bimod}
  Let $A$ be an involutive algebra and $V$ and involutive vector space.  Considering $\Aie$ as a
  $k[\Z/2]$-bimodule by the inclusion $k[\Z/2] \hookrightarrow \Aie$, we have an isomorphism of
  vector spaces
\[
\Aie \otimes_{\Z/2} V \cong A^e \otimes_k V,
\]
and under this identification, the involution on the left (coming from the left action of $\Z/2$ on
$\Aie$) corresponds with $(a\otimes b \otimes v)^* = b^*\otimes a^* \otimes v^*$.
\end{proposition}
\begin{proof}
We have an isomorphism of vector spaces, 
\[
\Aie\otimes_{\Z/2} V = A^e \otimes k[\Z/2]\otimes_{\Z/2} V \cong A^e \otimes_k V,
\]
and one easily checks that this is in fact an isomorphism of $A^e$-modules, i.e., $A$-bimodules.
This isomorphism is defined by sending
$(a\otimes b \otimes \tau)\otimes v^*= (a\otimes b\otimes 1)\otimes v$ to $(a\otimes b) \otimes v$,
for $a\otimes b \in A^e$ and $v\in V$.

It remains to examine the involution.  The involution on $\Aie\otimes_{\Z/2} V$, given by left
multiplication by $\tau$, is
\[
(a \otimes b \otimes \tau^i )\otimes v \mapsto (b^* \otimes a^* \otimes \tau^{i+1})\otimes v = (b^*
\otimes a^* \otimes \tau^i)\otimes v^*.
\]
Thus this corresponds to the involution $(a\otimes b) \otimes v \mapsto (b^* \otimes a^*) \otimes
v^*$ on $A^e\otimes_k V$.
\end{proof}

\begin{proposition}\label{prop:proj-bimod}
  Let $V$ be a projective involutive vector space.  The involutive bimodule $\Aie\otimes_{\Z/2}
  V$ is projective.
\end{proposition}
\begin{proof}
  We make use of the identification from Proposition \ref{prop:free-bimod}.
  Bimodule homomorphisms $f: A^e \otimes_k V  \to N$ are in bijection with linear maps
  $g: V\to N$ via the correspondence
\[
g(v) = f(1\otimes v \otimes 1), \text{ and } f(a\otimes v \otimes b) = ag(v)b.
\] 
Moreover, $f$ commutes with the involutions if and only if $g$ does.   Thus, to lift $f$ along a
surjection $M \to N$ of involutive bimodules, it suffices to produce a lift of $g$ in the category
of involutive vector spaces, and such a lift exists since $V$ is projective.
\end{proof}

An involutive bimodule $M$ is \emph{flat} if it is flat as an $\Aie$-module; equivalently,
it is flat if the functor $M\otimes_{\Aie} : A\mhyphen i\Bimod \to i\Vect$ is exact.  As
usual, if $M$ is a projective involutive bimodule then it is flat.

\subsection{The bar resolution as an involutive resolution}

First recall the classical bar resolution of an associative algebra $A$.  We write $Bar(A)$ for the
chain complex of bimodules whose degree $n$ part is $A^{\otimes_k (n+2)}$. This has the bimodule
structure given by
$a \cdot (a_0 \otimes \cdots \otimes a_{n+1}) \cdot b = aa_0 \otimes \cdots \otimes a_{n+1}b$ and in
particular, it is free, and hence projective as a bimodule.  The differential is defined by the
formula
\[
d(a_0\otimes \cdots \otimes a_{n+1}) = \sum_{i=0}^n (-1)^i a_0 \otimes \cdots \otimes a_i a_{i+1} \otimes
\cdots \otimes a_{n+1}.
\] 
The bar resolution of $A$ is augmented by the multiplication map $Bar_0(A) = A\otimes_k A \to A$.
Let $\Sigma A$ denote the graded vector space consisting of $A$ concentrated in degree $1$, and
write $T\Sigma A = \bigoplus_n (\Sigma A)^{\otimes n}$ for the tensor coalgebra with grading induced
from that of $\Sigma A$.  Regarding $T\Sigma A$ as a vector space, there is an isomorphism of graded
bimodules
\[
Bar(A) \stackrel{\cong}{\to} A^e \otimes_k T\Sigma A
\]
given by
$a_0 \otimes \cdots \otimes a_{n+1} \mapsto (a_0 \otimes a_{n+1}) \otimes (a_1\otimes \cdots \otimes
a_n)$.
I.e., $Bar(A)$ is the free graded $A$-bimodule generated by the underlying vector space $T\Sigma A$. 

Now suppose that $A$ is an involutive algebra.  In this case $T\Sigma A$ has an involution given by
$(a_1\otimes \cdots \otimes a_n)^* = a_n^* \otimes \cdots \otimes a_1^*$.  The bar resolution has an
involution given by
\[
(a_0 \otimes \cdots \otimes a_{n+1})^* = a_{n+1}^* \otimes \cdots \otimes a_0^*,
\]
and so we see that $Bar(A) \cong A^e \otimes_k T\Sigma A \cong \Aie\otimes_{\Z/2} T \Sigma
A$ is actually an isomorphism of involutive graded bimodules.  One can easily check that the
differential in $Bar(A)$ commutes with the involutions.  Hence we have:

\begin{proposition}\label{prop:resolution}
  If $A$ is an involutive algebra that is projective as an involutive vector space then the complex
  $Bar(A)$ is a projective resolution of $A$ as an involutive bimodule.
\end{proposition}
\begin{proof}
This follows directly from Proposition \ref{prop:proj-bimod}.
\end{proof}

\subsection{Involutive Hochschild homology and cohomology}

Involutive Hochschild cohomology has been defined in \cite{Braun} as the cohomology of the
complex of involution preserving coderivations (actually, he dualizes and then works with
derivations).  We instead define involutive Hochschild homology and cohomology as the derived
functors of involutive abelianization and involutive center.

We propose the following definitions.
\begin{definition}
The involutive Hochschild homology $iHH_*(A,M)$ of an involutive algebra $A$ with coefficients in an
involutive bimodule $M$ is the left derived functor of $iAb: A\mhyphen i\Bimod \to i\Vect$ evaluated
on $M$.  Similarly, the involutive Hochschild cohomology $iHH^*(A,M)$ is the right derived functor
of $iZ$ evaluated on $M$. 
\end{definition}

Equivalently,
\begin{align*}
HH_*(A,M) = \mathrm{Tor}_*^{\Aie}(A,M)\\
HH^*(A,M) = \mathrm{Ext}^*_{\Aie}(A,M).
\end{align*}

When $A$ is projective as an involutive vector space, then by Proposition \ref{prop:resolution} the
usual bar resolution in fact provides a resolution in the involutive setting, and so $iHH_*(A,M)$
and $iHH^*(A,M)$ can be computed by the complexes
\[
Bar(A)\otimes_{\Aie} M = iAb(Bar(A)\otimes_A M)
\]
and
\[
\ihomA(Bar(A),M)
\]
respectively.

The standard Hochschild chain complex, $C_*(A,M)$, is the abelianization of the $A$-bimodule
$Bar(A)\otimes_A M$, or equivalently it is $Bar(A)\otimes_{A^e} M$, and this has the familiar
description
\begin{equation}\label{eq:usual-complex}
C_n(A,M) \cong A^{\otimes n} \otimes_k M,
\end{equation}
with differential 
\begin{align*}
d: a_1\otimes \cdots \otimes a_n \otimes m \quad \mapsto  & \quad a_2 \otimes \cdots \otimes a_n \otimes ma_1 \\
&+ \sum^n_{i=1} (-1)^i a_1 \otimes \cdots \otimes a_i a_{i+1} \otimes \cdots \otimes a_n \\
& + (-1)^n a_1 \otimes\cdots \otimes a_{n-1}\otimes a_n m.
\end{align*}

If $A$ is an involutive algebra (that is projective as an involutive vector space) and $M$ is an
involutive bimodule then the involutive Hochschild homology is computed by the complex
$Bar(A)\otimes_{\Aie} M$, and by Proposition \ref{prop:Ae-to-Aie-quotient}, this is the
$\Z/2$-coinvariants of $Bar(A)\otimes_{A^e} M$ for the action given by
$a_0 \otimes \cdots \otimes a_{n+1} \otimes m \mapsto a_{n+1}^* \otimes \cdots \otimes a_n^* \otimes
m^*$.
Under the identification in \eqref{eq:usual-complex}, the $\Z/2$-action on $Bar(A) \otimes_A M$
corresponds to
\[
a_1 \otimes \cdots \otimes a_n \otimes m  \mapsto a_n^* \otimes \cdots \otimes a^*_1 \otimes m^*.
\]
We thus have the following result.

\begin{proposition}
If $A$ is an involutive algebra that is projective as an involutive vector space, and $M$ is an
involutive bimodule, then $iHH_*(A,M)$ is computed by the complex $iC_*(A,M)$ defined as
\begin{align*}
iC_n(A,M) & = A^{\otimes n} \otimes M / (a_1 \otimes \cdots \otimes a_n \otimes m  -  a_n^* \otimes
\cdots \otimes a^*_1 \otimes m^*) \\
& \cong C_n(A,M)_{\Z/2},
\end{align*}
with differential induced by the usual Hochschild differential.
\end{proposition}

The $\Z/2$ action on $C_*(A,M)$ induces an action on $HH_*(A,M)$. 
\begin{proposition}
  If the characteristic of the ground field $k$ is different from 2 then
\[
iHH_*(A,M) = HH_*(A,M)_{\Z/2}.
\]
\end{proposition}
\begin{proof}
This is immediate from the fact that taking $\Z/2$-coinvariants is exact when the characteristic is
not equal to 2.
\end{proof}

\section{Comparison with Braun's definition}

We now compare our definition of involutive Hochschild cohomology with Braun's definition and show
that they agree when for involutive algebras that are projective as involutive vector spaces.  

\subsection{Derivations and coderivations}
Given a graded algebra $A$, let $\mathrm{Der}(A)$ denote the space of graded derivations of $A$ into
itself, and given a coalgebra $C$, let $\mathrm{Coder}(C)$ denote the space of coderivations of $C$
into itself.  
If $A$ is an involutive algebra then we write $i\mathrm{Der}(A)$ for the subspace of
involution-preserving derivations, and likewise for the notation $i\mathrm{Coder}(C)$ if $C$ carries
an involution.  The spaces of derivations and coderivations are graded Lie algebras with respect to
the commutator bracket, and the subspaces $i\mathrm{Der}(A)$ and $i\mathrm{Coder}(A)$ are Lie
subalgebras.

If $V$ is a graded involutive vector space then the tensor algebra $TV$ carries an involution given
by
\[
(v_1\otimes \cdots \otimes v_n)^* = v_n^* \otimes \cdots \otimes v_1^*.
\]
If $A$ is an associative algebra then the multiplication map induces a coderivation $m$ on
$T\Sigma A$ of degree $-1$ and a derivation $m'$ on $\widehat{T}\Sigma^{-1} A^\vee$ also of degree
$-1$.  If $A$ is an involutive algebra then $m$ and $m'$ are both involution-preserving.  The
commutator $[m,-]$ defines a differential on $i\mathrm{Coder}(T\Sigma A)$, and likewise for the space
$i\mathrm{Der}(\widehat{T}\Sigma^{-1} A^\vee)$ of continuous involution-preserving derivations on
the completed tensor algebra.

Braun defines the involutive Hochschild cohomology to be the cohomology computed by the complex
\[
\Sigma^{-1} i\mathrm{Der}(\widehat{T}\Sigma^{-1}A^\vee).
\]
Since $T\Sigma A$ dualizes to $\widehat{T}\Sigma^{-1}A^\vee$, algebra derivations on the latter are
the same as coalgebra coderivations on the former, and hence there is an isomorphism of complexes
\[
\Sigma^{-1} i\mathrm{Der}(\widehat{T}\Sigma^{-1}A^\vee) \cong \Sigma^{-1} i \mathrm{Coder}(T\Sigma A).
\]

As we have seen, the bar resolution provides a resolution in the involutive category, and this next proposition
confirms that our derived functor definition of Hochschild cohomology agrees with Braun's definition.
\begin{proposition}
There is a canonical isomorphism of complexes,
\[
\ihomA(Bar(A),A) \cong \Sigma^{-1} i \mathrm{Coder}(T\Sigma A).
\]
\end{proposition}
\begin{proof}
  Since $Bar(A) \cong A\otimes_k T\Sigma A \otimes_k A$, the degree $n$ part of
  $\Hom_{A-\Bimod}(Bar(A),A)$ is the space of degree $-n$ linear maps $T\Sigma A \to A$, which is
  isomorphic to the space of degree $(-n-1)$ linear maps $T\Sigma A \to \Sigma A$.  By the universal
  property of the tensor coalgebra, there is a bijection between degree $(-n-1)$ linear maps
  $T\Sigma A \to \Sigma A$ and degree $(-n-1)$ coderivations on $T\Sigma A$.  Hence the degree $n$
  part of $\Hom_{A-\Bimod}(Bar(A),A)$ is isomorphic to the degree $n$ part of
  $\Sigma^{-1} \mathrm{Coder}(T\Sigma A)$.  One now checks directly that this isomorphism restricts
  to an isomorphism of graded vector spaces
\[
i\Hom_{A-\Bimod}(Bar(A),A) \cong \Sigma^{-1} i\mathrm{Coder}(T\Sigma A).
\]
With a bit of tedious but straightforward algebra one can check that
the differentials coincide under the above isomorphism, cf.~\cite[\S12.2.4]{Loday-Vallette}.  
\end{proof}









  

\bibliographystyle{amsalpha}
\bibliography{bib}

\end{document}

%% file: preamble.tex
\usepackage[foot]{amsaddr}

\usepackage{mathptmx}
\DeclareSymbolFont{cmlargesymbols}{OMX}{cmex}{m}{n}
\DeclareMathSymbol{\mycoprod}{\mathop}{cmlargesymbols}{"60}

\usepackage[mathcal]{eucal}
\DeclareFontFamily{OT1}{pzc}{}
\DeclareFontShape{OT1}{pzc}{m}{it}{<-> s * [1.10] pzcmi7t}{}
\DeclareMathAlphabet{\mathpzc}{OT1}{pzc}{m}{it}

\usepackage[protrusion=true,expansion=true]{microtype}

\usepackage{amssymb}
\usepackage{mathrsfs}  
\usepackage{latexsym}
\usepackage{amsmath}
\usepackage{color}
\usepackage{hyperref}

\usepackage[cmtip,arrow]{xy}
\usepackage{pb-diagram,pb-xy}
\usepackage[vmargin=3cm, hmargin=3cm]{geometry}
\numberwithin{equation}{subsection}

\newtheorem{theorem}{Theorem}[subsection]  

\newtheorem{proposition}[theorem]{Proposition}

\theoremstyle{remark} 
\newtheorem{definition}[theorem]{Definition}
\newtheorem{remark}[theorem]{Remark}
\newtheorem{example}[theorem]{Example}

\newcommand{\Z}{\mathbb{Z}}

\newcommand{\Hom}{\mathrm{Hom}}

\newcommand{\Vect}{\mathpzc{Vect}}
\newcommand{\Mod}{\mathpzc{Mod}}
\newcommand{\Bimod}{\mathpzc{Bimod}}

\mathchardef\mhyphen="2D

